\documentclass[a4paper,12pt,reqno]{amsart}

\newtheorem{theorem}{Theorem}[section]
\newtheorem{lemma}[theorem]{Lemma}

\theoremstyle{definition}

\theoremstyle{remark}

\newcommand{\pt}{\partial}
\newcommand{\rre}{{\mathbb R}}
\newcommand{\cmx}{{\mathbb C}}
\newcommand{\ttt}{{\mathbb T}}
\newcommand{\zzz}{{\mathbb Z}}
\newcommand{\pe}{\psi_{\epsilon}}

\numberwithin{equation}{section}

\begin{document}

\title[fourth order
nonlinear Schr\"odinger type
equation]{Refined Energy Inequality with 
Application to Well-posedness for the Fourth Order\\
Nonlinear Schr\"odinger Type\\
Equation on Torus}

\author{Jun-ichi Segata}
\address{Mathematical Institute, Tohoku University, Aoba, Sendai 980-8578, Japan}
\email{segata@math.tohoku.ac.jp}
\thanks{The author is partially 
supported by MEXT, 
Grant-in-Aid for Young Scientists 
(B) 21740122.}


\date{}



\begin{abstract}
We consider the time local and global 
well-posedness for 
the fourth order nonlinear 
Schr\"{o}dinger type equation (4NLS) on the torus. 
The nonlinear term of (4NLS) contains the derivatives 
of unknown function and 
this prevents us to apply the classical energy method. 
To overcome this difficulty, we introduce the modified energy 
and derive an a priori estimate for the solution to (4NLS).  
\end{abstract}

\maketitle

\section{Introduction} \label{sec:intro}

We consider the fourth order nonlinear 
Schr\"{o}dinger type equation (4NLS) 
on the torus $\ttt=\rre/2\pi\zzz$:
\begin{eqnarray}
\left\{
\begin{array}{l}
\displaystyle{
i\pt_t\psi+\pt_x^2\psi+\nu\pt_x^4\psi
={{\mathcal N}}(\psi,\overline{\psi},
\pt_x\psi,\pt_x\overline{\psi},
\pt_x^2\psi,\pt_x^2\overline{\psi}),}\\
\displaystyle{\psi(0,x)=\phi(x),\qquad x\in\ttt,}
\end{array}
\right.
\label{4NLS}
\end{eqnarray}
where $\pt_t=\pt/\pt t$, $\pt_x=\pt/\pt x$, 
$\psi:\rre\times\ttt\to\cmx$ is an unknown 
function, and $\phi:\ttt\to\cmx$ is a given  
function. The nonlinear term ${{\mathcal N}}$ 
is given by
\begin{eqnarray*}
{{\mathcal N}}(\psi,\overline{\psi},
\ldots, 
\pt_x^2\psi,\pt_x^2\overline{\psi})
&=&
\lambda_1|\psi|^2\psi+\lambda_2|\psi|^4\psi
+\lambda_3(\pt_x\psi)^2\overline{\psi}
+\lambda_4|\pt_x\psi|^2\psi\\
& &+\lambda_5\psi^2\pt_x^2\overline{\psi}
+\lambda_6|\psi|^2\pt_x^2\psi,
\end{eqnarray*}
where $\nu\neq0$ and $\lambda_j,\ j=1,\cdots,6$ are 
real constants. 
The equation (\ref{4NLS}) arises in the context of 
a motion of vortex filament. 
More precisely, using the localized induction approximation, 
Da Rios \cite{Da} proposed 
some equation which approximates the three 
dimensional motion of an isolated vortex 
filament embedded in an inviscid incompressible 
fluid filling an infinite region. The Da Rios equation 
is reduced to the cubic nonlinear Schr\"{o}dinger 
equation
\begin{eqnarray}
i\pt_t\psi+\pt_x^2\psi=-\frac12|\psi|^2\psi, 
\quad(t,x)\in\rre\times\ttt
\label{NLS}
\end{eqnarray}
via the Hasimoto transform \cite{H}. 
To describe the motion of actual vortex filament more precisely, some detailed models taking 
into account the effect from higher order corrections of equation have been introduced 
by Fukumoto-Moffatt \cite{FM}. The Fukumoto-Moffatt equation is 
rewritten as (\ref{4NLS}) by using the Hasimoto transform. 
For the physical background of (\ref{4NLS}), see Fukumoto-Moffatt 
\cite{FM}.

In this paper we consider the time local well-posedness for (\ref{4NLS}) 
on the Sobolev spaces $H^m(\ttt)$. 
Our notion of well-poseness contains the 
existence and uniqueness of the solution and 
the continuity of the data-to-solution map. 
We also consider the persistent property 
of the solution, that is, the solution describes a continuous curve 
in $H^m(\ttt)$ whenever $\phi\in H^m(\ttt)$. 
Our motivation to consider the time local well-posedness for (\ref{4NLS}) is 
that we are interested in the stability of the 
standing wave solution $\psi(t,x)=e^{i\omega t}\varphi_{\omega}(x)$ 
to (\ref{4NLS}). When (\ref{4NLS}) is completely integrable (see the later half 
of this section below for the detail), 
(\ref{4NLS}) has the sech-type standing wave solution. 
The orbital stability in $H^m(\rre)$ of the 
sech-type standing wave solution 
is proved by \cite{MS}. On the other hand,  
we easily see that (\ref{4NLS}) has a exact periodic standing wave 
solution of the form 
$\psi(t,x)=\kappa e^{i\tau x+i\omega t}$ for some real constants $\kappa,\tau$ 
and $\omega$. It is interesting that the sech-type standing wave 
and the periodic standing wave 
correspond to the tornado like curve and the helicoid curve 
in the motion of the vortex filament, see Kida \cite{Kida}. 

As the first step to show the orbital stability of 
the sech-type and the periodic standing wave, we need to prove the 
global well-posedness for (\ref{4NLS}) in the Sobolev spaces 
on the real line $\rre$ and 
on the torus $\ttt$, respectively. 
Concerning the local well-posedness of (\ref{4NLS}) 
on real line $\rre$,  
Segata \cite{S1,S2,S3} and Huo-Jia \cite{HJ1,HJ2} proved 
that the initial value 
problem of (\ref{4NLS}) is locally well-posed in 
Sobolev space $H^s(\rre)$ with $s>1/2$  
by using the Fourier restriction method 
introduced by Bourgain \cite{B} and  
Kenig-Ponce-Vega \cite{KPV2,KPV3}. 
As far as we know, 
there is no result on the well-posedness 
of (\ref{4NLS}) under the periodic boundary condition. 

In this paper we focus on the well-posedness of 
(\ref{4NLS}) on the torus. There is a large literature on 
the well-posedness for the dispersive equations in the 
torus. See for instance \cite{D,KPV1,M,MV} for the linear dispersive equations 
and \cite{ABFS,B,C,GH,H,Sch,TF1,TF2} 
for the non-linear dispersive equations. 
We summarize the well-posedness on the 
derivative nonlinear Schr\"{o}dinger equation 
with the periodic boundary condition. 
Tsutsumi-Fukuda \cite{TF1,TF2} proved the 
local and global 
well-posedness for the Schr\"{o}dinger  
equation with some nonlinearity on the torus 
by using the classical energy method. 
Gr\"{u}nrock-Herr \cite{GH} and Herr \cite{H} obtained 
sharp well-posedness results for some 
derivative nonlinear Schr\"{o}dinger  
equation on the torus by using the Fourier restriction method. 
The well-posedness 
of the Schr\"{o}dinger equation for more general derivative 
nonlinearity in the $n$-dimensional torus was given by 
Chihara \cite{C}. 
We notice that the classical energy method 
does not works for his setting. In \cite{C} 
he conquered this problem by using the 
pseudo-differential operators with non-smooth 
coefficients on the torus. 

As we shall see below, the dispersive equations on the torus 
do not have fine properties compared to the real line case. 
Therefore the proof of the well-posedness on the torus 
become increasingly harder than the real line case. 
To state our results more precisely, we introduce several notations. 
Given a function $\psi$ on $\ttt$, we define the Fourier 
coefficient of $\psi$, by
\begin{eqnarray*}
\hat{\psi}(n)=\frac{1}{\sqrt{2\pi}}\int_0^{2\pi}
\psi(x)e^{-inx}dx,\qquad n\in\zzz.
\end{eqnarray*}
Let $m$ be a non-negative integer. $H^m(\ttt)$ 
denotes the all tempered distributions on $\ttt$ 
satisfying 
\begin{eqnarray*}
\|\psi\|_{H_x^m}=(\sum_{n\in\zzz}\langle n\rangle^{2m}
|\hat{\psi}(n)|^2)^{1/2}<+\infty, 
\end{eqnarray*}
where $\langle n\rangle=\sqrt{1+n^2}$. 

The main result in this paper is the following:

\begin{theorem}\label{well-posedness} Let $m\ge4$ be 
an integer. Then (\ref{4NLS}) is locally well-posed 
in the following sense:
For any $\phi\in H^m(\ttt)$, there exists a time 
$T=T(\|\phi\|_{H^m})>0$ and a unique solution 
$\psi$ of (\ref{4NLS}) satisfying
\begin{eqnarray*}
\psi\in C([0,T);H^m(\rre)).
\end{eqnarray*} 
Moreover, the data-to-solution map 
$H^m(\ttt)\rightarrow C([0,T];H^m(\ttt)) 
(\phi\mapsto\psi(t))$ is continuous.
\end{theorem}

The difficulty in the proof of time local well-posedness 
of (\ref{4NLS}) arises in so called ``loss of a derivatives". 
More precisely, the standard energy estimate gives only the following: 
\begin{eqnarray}
\lefteqn{\frac{d}{dt}\|\pt_x^m\psi(t)\|_{L_x^2}^2}\label{e2}\\
&=&2\{2\lambda_3+\lambda_4+2(m-1)\lambda_6\}
\mbox{Im}\int_{\ttt}\overline{\psi}\pt_x\psi\cdot\pt_x^m\overline{\psi}\pt_x^{m+1}\psi dx
\nonumber\\
& &-2\lambda_5\mbox{Im}
\int_{\ttt}\psi^2(\pt_x^{m+1}\overline{\psi})^2dx
+\ell.o.t.
\nonumber
\end{eqnarray}
Since the first and second terms in the 
right hand side of (\ref{e2}) contain the $(m+1)$-st derivatives 
of $\psi$, we cannot control those factors in terms of $H^m$ norm of $\psi$.
Therefore this estimate does not give an a priori estimate for the solution. 

For the real line case, 
the unitary group 
$\{e^{it(\pt_x^2+\nu\pt_x^4)}\}_{t\in\rre}$ generated by
the linear operator $i\pt_x^2+i\nu\pt_x^4$ 
gains extra smoothness in space variable, see 
Kenig-Ponce-Vega \cite{KPV1}. 
Thanks to this smoothing property for 
$\{e^{it(\pt_x^2+\nu\pt_x^4)}\}_{t\in\rre}$, in 
\cite{S1,HJ1,S2,HJ2} they could overcome 
a loss of derivatives and 
guarantee the well-posedness of (\ref{4NLS}) 
on $\rre$. However, for the periodic case 
the corresponding unitary group  does not 
have such a fine properties (see e.g., \cite{D}) 
and it is not likely that the contraction mapping 
principle guarantees the well-posedness for 
(\ref{4NLS}) on $\ttt$. Since this is the case 
we abandon making use of the property of the unitary group 
$\{e^{it(\pt_x^2+\nu\pt_x^4)}\}_{t\in\rre}$  
and try to this issue by a different approach. 

Let us return the estimate (\ref{e2}). If we contrive to eliminate the worst terms, 
we can obtain an a priori estimate of solution. 
In this paper we take a hint from Kwon \cite{Kwon} 
which is concerned with the well-posedness for the fifth-order KdV equation on $\rre$, 
we introduce the``modified" energy:
\begin{eqnarray*}
\lefteqn{[E_m(\psi)](t)}\\
&=&\|\pt_x^m\psi(t)\|_{L_x^2}^{2}
+\|\psi(t)\|_{L_x^2}^2
+C_m\|\psi(t)\|_{L_x^2}^{4m+2}\\
& &
+\frac{\lambda_5}{\nu}\mbox{Re}\int_{\ttt}(\pt_x^{m-1}\psi)^2
\overline{\psi}^2dx
+\frac{2\lambda_3+\lambda_4+2(m-1)\lambda_6}{4\nu}
\int_{\ttt}|\pt_x^{m-1}\psi|^2|\psi|^2dx,
\end{eqnarray*}
where $C_m$ is a sufficiently large constant depending only on 
$m$ so that $E_m(\psi)$ is positive. Thanks  to the correction terms 
we can eliminate the worst factors in $(\ref{e2})$ and 
evaluate the $H^m$ norm of the solution $\psi$ 
to (\ref{4NLS}) in terms of the 
$H^m$ norm of the initial data 
$\phi$. This is a crucial point in the proof of Theorem \ref{well-posedness}. 

It is known that $(\ref{4NLS})$ is completely 
integrable if and only if $\lambda_1=-1/2$, 
$\lambda_2=-3\nu/8,
\lambda_3=-3\nu/2$, 
$\lambda_4=-\nu, \lambda_5=-\nu/2$ and 
$\lambda_6=-2\nu$. In this case 
(\ref{4NLS}) has infinitely many  
conservation quantities, see Langer and Perline \cite{LP}. 
The first three conservation quantities for (\ref{4NLS}) are given by 
\begin{eqnarray*}
I_0(\psi)&=&\frac12\int_{\ttt}|\psi|^2dx,\\
I_1(\psi)&=&\frac12\int_{\ttt}|\pt_x\psi|^2dx
-\frac18\int_{\ttt}|\psi|^4dx,\\
I_2(\psi)&=&
\frac12\int_{\ttt}|\pt_x^2\psi|^2dx
+\frac34\int_{\ttt}|\psi|^2\overline{\psi}\pt_x^2\psi dx
+\frac18\int_{\ttt}|\psi|^2\psi\pt_x^2\overline{\psi}dx,\\
& &+\frac58\int_{\ttt}(\pt_x\psi)^2
\overline{\psi}^2dx
+\frac34\int_{\ttt}|\pt_x\psi|^2|\psi|^2dx
+\frac{1}{16}\int_{\ttt}|\psi|^6dx.
\end{eqnarray*}
In general, the conservation quantities 
for (\ref{4NLS}) are expressed as 
\begin{eqnarray*}
I_m(\psi)=\frac12\int_{\ttt}|\pt_x^m\psi|^2+
\int_{\ttt}Q_m(\psi,\overline{\psi},
\ldots, 
\pt_x^{m-1}\psi,\pt_x^{m-1}\overline{\psi})dx,
\end{eqnarray*}
where $Q_m$ are some polynomials in 
$(\psi,\overline{\psi},
\ldots, 
\pt_x^{m-1}\psi,\pt_x^{m-1}\overline{\psi})$ satisfying
the inequalities $|Q_m|\le C_m\|\psi\|_{L_x^2}^{\alpha_m}
\|\pt_x^m\psi\|_{L_x^2}^{\beta_m}$ for some $\alpha_m>0$ 
and $0<\beta_m<2$. 
Therefore combining Theorem \ref{well-posedness}, 
the conservation laws $I_m(\psi)(t)=I_m(\psi)(0)$ 
and Young's inequality, 
we obtain the 
global existence theorem for (\ref{4NLS}) in $H^m(\ttt)$:

\begin{theorem} 
Assume $\lambda_1=-1/2$, 
$\lambda_2=-3\nu/8,
\lambda_3=-3\nu/2$, 
$\lambda_4=-\nu, \lambda_5=-\nu/2$ and 
$\lambda_6=-2\nu$. Then 
(\ref{4NLS}) is globally well-posed in $H^m(\ttt)$  
with an integer $m$ greater than 3.
\end{theorem}

Finally we point out that by combining our proof 
with the estimates for the fractional derivatives 
we may well be able to 
extend Theorem \ref{well-posedness} to 
the case where $m$ is not an integer . 
In this paper we do not touch on this issue.  

The plan of this paper is as follows. 
Section 2 is devoted to the parabolic regularization  
associated to (\ref{4NLS}). In Section 3, 
we introduce the modified energy and give an 
a priori estimate for the solution to (\ref{4NLS}). 
Then we shall  
prove the existence of solution to (\ref{4NLS}). 
In Section 4 we give the proofs of the uniqueness and 
the persistent properties of solution to (\ref{4NLS}), and 
the continuous dependence of the solution to (\ref{4NLS}) 
on the initial data. 

\section{Parabolic Regularization} \label{sec:parabolic}

In this section, we consider the parabolic regularization of 
(\ref{4NLS}) in $H^m(\ttt)$. 
We first give the Gagliardo-Nirenberg inequality 
for the periodic functions. 

\begin{lemma}\label{gn} 
Let $l$ and $m$ be integers satisfying $0\le l\le m-1$ 
and let $2\le p\le\infty$. 
Then there exists a constant $C$ depending only on $l,m$ 
and $p$ 
such that for any $\psi\in H^m(\ttt)$, 
\begin{eqnarray*}
\|\pt_x^l\psi\|_{L_x^p}\le
C\times\left\{
\begin{array}{l}
\displaystyle{\|\psi\|_{L_x^2}^{1-\alpha}\|\pt_x^m\psi\|_{L_x^2}^{\alpha}
\qquad\qquad(1\le l\le m-1),
}\\
\displaystyle{\|\psi\|_{L_x^2}^{1-\alpha}\|\pt_x^m\psi\|_{L_x^2}^{\alpha}
+\|\psi\|_{L_x^2}\qquad(l=0),}
\end{array}
\right.
\end{eqnarray*}
where $\displaystyle{\alpha=(l+1/2-1/p)/m}$. 
Especially, we have $\|\pt_x^l\psi\|_{L_x^p}\le
C\|\psi\|_{L_x^2}^{1-\alpha}\|\psi\|_{H_x^m}^{\alpha}$.
\end{lemma}

\begin{proof}
See for instance, \cite[Section 2]{Sch}. 
\end{proof}

Next we consider the parabolic regularization of 
(\ref{4NLS}). Let us introduce the regularizing 
sequence used in Bona-Smith \cite{BS}. Let $\varphi\in C^{\infty}(\rre)$ 
be such that $0\le\varphi(\xi)\le1$ for $\xi\in\rre$, $\varphi^{(k)}(0)=0$ 
for $k\in{\mathbb N}$ and $\varphi(\xi)$ tends exponentially to $0$ as 
$|\xi|\to\infty$. We define for $\epsilon\in(0,1]$, 
\begin{eqnarray*}
\phi_{\epsilon}(x)=\frac{1}{\sqrt{2\pi}}
\sum_{n\in{\mathbb Z}}\varphi(\epsilon n)\hat{\phi}(n)e^{inx}.
\end{eqnarray*}
Then, $\{\phi_{\epsilon}\}_{\epsilon>0}\in H^{\infty}(\ttt)$ and 
$\|\phi-\phi_{\epsilon}\|_{H_x^m}\to0$ as $\epsilon\to0$. Furthermore, 
for any $l\ge0$, 
\begin{eqnarray*}
\|\phi_{\epsilon}\|_{H_x^{m+l}}&\le& C\epsilon^{-l}\|\phi\|_{H_x^m},\\
\|\phi-\phi_{\epsilon}\|_{H_x^{m}}&\le& C\|\phi\|_{H_x^m},\\
\|\phi-\phi_{\epsilon}\|_{H_x^{m-l}}&\le& C\epsilon^{l}\|\phi\|_{H_x^m}.
\end{eqnarray*}
We consider the regularized problem of (\ref{4NLS}):
\begin{eqnarray}
\left\{
\begin{array}{l}
\displaystyle{
i\pt_t\pe+\pt_x^2\pe+(\nu+i\epsilon)\pt_x^4\pe
={{\mathcal N}}(\pe,\overline{\psi}_{\epsilon},
\ldots,\pt_x^2\pe,\pt_x^2\overline{\psi}_{\epsilon}),
}\\
\displaystyle{\pe(0,x)=\phi_{\epsilon}(x),
}
\end{array}
\right.
\label{3.1}
\end{eqnarray}
where $\psi_{\epsilon}:\rre\times\ttt\to\cmx$ is an unknown 
function, and $\phi_{\epsilon}:\ttt\to\cmx$ is a 
Bona-Smith approximation of $\phi$. 
Concerning the solvability of (\ref{3.1}), we have the following lemma.

\begin{lemma}\label{l3.1}
Let $m\ge3$ be an integer. For any $\phi\in H^m(\ttt)$, 
there exists a time $T_{\epsilon}=T(\epsilon,\|\phi\|_{H^m})>0$ 
and a unique solution $\pe$ of (\ref{3.1}) satisfying 
\begin{eqnarray*}
\pe\in C([0,T_{\epsilon}),H^m(\ttt)).
\end{eqnarray*}
\end{lemma}

\begin{proof}
We shall prove (\ref{3.1}) by using the Banach fixed point theorem. 
Let $\{W_{\epsilon}(t)\}_{t\ge0}$ be the contraction semi-group generated by 
the linear operator $i\pt_x^2+i\nu\pt_x^4-\epsilon\pt_x^4$:
\begin{eqnarray*}
[W_{\epsilon}(t)\phi](x)
=\frac{1}{\sqrt{2\pi}}\sum_{n\in\zzz}
\hat{\phi}(n)e^{inx+(-in^2+i\nu n^4-\epsilon n^4)t}.
\end{eqnarray*}
Then, the initial value problem (\ref{3.1}) is rewritten as 
the integral equation
\begin{eqnarray*}
\pe(t)=W_{\epsilon}(t)\phi_{\epsilon}
-i\int_0^tW_{\epsilon}(t-\tau){{\mathcal N}}(\psi_{\epsilon},
\overline{\psi}_{\epsilon},
\ldots,\pt_x^2\psi_{\epsilon},\pt_x^2\overline{\psi}_{\epsilon})(\tau)d\tau.
\end{eqnarray*}
We put $r=\|\phi\|_{H_x^m}$. For $T>0$, we define 
\begin{eqnarray*}
X_T^r=\{\psi\in C([0,T];H^m(\ttt))|\sup_{t\in[0,T)}\|\psi(t)\|_{H_x^m}\le 2r\}. 
\end{eqnarray*}
We shall show that the map 
\begin{eqnarray*}
\Phi(\pe)=W_{\epsilon}(t)\phi_{\epsilon}
-i\int_0^tW_{\epsilon}(t-\tau){{\mathcal N}}(\psi_{\epsilon},
\overline{\psi}_{\epsilon},
\ldots,\pt_x^2\psi_{\epsilon},\pt_x^2\overline{\psi}_{\epsilon})(\tau)d\tau
\end{eqnarray*}
is a contraction on $X_T^r$ for choosing $T$ suitably.

We easily see that 
\begin{eqnarray}
\lefteqn{\|\Phi(\pe)(t)\|_{H_x^m}}\label{3.2}\\
&\le&\|\phi\|_{H_x^m}+
\int_0^t\|W_{\epsilon}(t-\tau){{\mathcal N}}(\psi_{\epsilon},
\overline{\psi}_{\epsilon},
\ldots,\pt_x^2\psi_{\epsilon},\pt_x^2\overline{\psi}_{\epsilon})(\tau)\|_{H_x^m}d\tau.
\nonumber
\end{eqnarray}
By Plancherel's identity, we obtain
\begin{eqnarray*}
\lefteqn{\int_0^t\|W_{\epsilon}(t-\tau){{\mathcal N}}
(\psi_{\epsilon},\overline{\psi}_{\epsilon},
\ldots,\pt_x^2\psi_{\epsilon},
\pt_x^2\overline{\psi}_{\epsilon})(\tau)\|_{H_x^m}d\tau}\\
&=&
\int_0^t\left\{\sum_{n\in\zzz}\langle n\rangle^{2m}
|\widehat{{{\mathcal N}}}(\psi_{\epsilon}
,\overline{\psi}_{\epsilon},
\ldots,\pt_x^2\psi_{\epsilon}
,\pt_x^2\overline{\psi}_{\epsilon})(\tau,n)
e^{(-in^2+i\nu n^4-\epsilon n^4)(t-\tau)}|^2\right\}^{1/2}d\tau\\
&=&
\int_0^t\left\{\sum_{n\in\zzz}\langle n\rangle^{2m-4}
|\widehat{{{\mathcal N}}}(\psi_{\epsilon}
,\overline{\psi}_{\epsilon},
\ldots,\pt_x^2\psi_{\epsilon}
,\pt_x^2\overline{\psi}_{\epsilon})(\tau,n)|^2
\langle n\rangle^4
e^{-2\epsilon n^4(t-\tau)}\right\}^{1/2}d\tau\\
&\le&
\int_0^t\sup_{n\in\zzz}\{\langle n\rangle^2
e^{-\epsilon n^4(t-\tau)}\}\|{{\mathcal N}}(\psi_{\epsilon}
,\overline{\psi}_{\epsilon},
\ldots,\pt_x^2\psi_{\epsilon}
,\pt_x^2\overline{\psi}_{\epsilon})(\tau)\|_{H_x^{m-2}}d\tau.
\end{eqnarray*}
Since $\displaystyle{\sup_{n\in\zzz}\{\langle n\rangle^2
e^{-\epsilon n^4(t-\tau)}\}\le1+\epsilon^{-1/2}(t-\tau)^{-1/2}}$, 
\begin{eqnarray}
\lefteqn{\int_0^t\|W_{\epsilon}(t-\tau){{\mathcal N}}
(\psi_{\epsilon},\overline{\psi}_{\epsilon},
\ldots,\pt_x^2\psi_{\epsilon},\pt_x^2\overline{\psi}_{\epsilon})
(\tau)\|_{H_x^m}d\tau}
\label{3.4}\\
&\le&C\int_0^t\{1+\epsilon^{-1/2}(t-\tau)^{-1/2}\}\|{{\mathcal N}}
(\psi_{\epsilon},\overline{\psi}_{\epsilon},
\ldots,\pt_x^2\psi_{\epsilon}
,\pt_x^2\overline{\psi}_{\epsilon})(\tau)\|_{H_x^{m-2}}d\tau
\nonumber\\
&\le&C(t+\epsilon^{-1/2}t^{1/2})\sup_{t\in[0,T)}\|{{\mathcal N}}
(\psi_{\epsilon},\overline{\psi}_{\epsilon},
\ldots,\pt_x^2\psi_{\epsilon},
\pt_x^2\overline{\psi}_{\epsilon})(t)\|_{H_x^{m-2}}.
\nonumber
\end{eqnarray}
Collecting (\ref{3.2}) and (\ref{3.4}), we have
\begin{eqnarray*}
\lefteqn{\sup_{t\in[0,T)}\|\Phi(\pe)(t)\|_{H_x^m}}\\
&\le&\|\phi\|_{H_x^m}+C(T+\epsilon^{-1/2}T^{1/2})
\sup_{t\in[0,T)}\|{{\mathcal N}}(\psi_{\epsilon},
\overline{\psi}_{\epsilon},
\ldots,\pt_x^2\psi_{\epsilon},\pt_x^2\overline{\psi}_{\epsilon})
(t)\|_{H_x^{m-2}}.
\end{eqnarray*}
By Sobolev's embedding, we have
\begin{eqnarray*}
\lefteqn{\sup_{t\in[0,T)}\|{{\mathcal N}}(\psi_{\epsilon},\overline{\psi}_{\epsilon},
\cdots,\pt_x^2\psi_{\epsilon},\pt_x^2\overline{\psi}_{\epsilon})(t)\|_{H_x^{m-2}}}\\
&\le&C(1+\sup_{t\in[0,T)}\|\psi(t)\|_{H_x^{m}}^2)
\sup_{t\in[0,T)}\|\psi(t)\|_{H_x^{m}}^3.
\end{eqnarray*}
Therefore, 
\begin{eqnarray*}
\sup_{t\in[0,T)}\|\Phi(\pe)(t)\|_{H_x^m}
\le r+C(T+\epsilon^{-1/2}T^{1/2})
(1+r^2)r^3.
\end{eqnarray*}
We can easily check that $\Phi(\pe)\in C([0,T);H^m(\ttt))$. 
Therefore, by choosing $T_{\epsilon}>0$ sufficiently 
small so that $C(T_{\epsilon}+\epsilon^{-1/2}T_{\epsilon}^{1/2})
(1+r^2)r^2<1$ we have $\pe\in X_T^r$. 
By a similar way, for $\pe^1,\pe^2\in X_T^r$, we have
\begin{eqnarray*}
\lefteqn{\sup_{t\in[0,T)}\|\Phi(\pe^1)(t)-\Phi(\pe^2)(t)\|_{H_x^m}}\\
&\le&C(T+\epsilon^{-1/2}T^{1/2})
(1+r^2)r^2\sup_{t\in[0,T)}\|\pe^1(t)-\pe^2(t)\|_{H_x^m}\\
&<&\sup_{t\in[0,T)}\|\pe^1(t)-\pe^2(t)\|_{H_x^m}.
\end{eqnarray*}
Consequently, we have that $\Phi$ is a contraction on $X_T^r$. 
The Banach fixed point theorem implies the unique existence 
of solution to (\ref{3.1}) in $X_T^r$ which completes the 
proof of Lemma \ref{l3.1}. 
\end{proof}

\section{Modified Energy} \label{sec:energy}

In this section, by using the modified energy, we give an a priori 
estimates for the solution to (\ref{3.1}) obtained by Lemma \ref{l3.1}.

Let $m\ge1$ be an integer. We introduce the modified energy:
\begin{eqnarray*}
\lefteqn{[E_m(\psi)](t)}\\
&=&\|\pt_x^m\psi(t)\|_{L_x^2}^{2}+\|\psi(t)\|_{L_x^2}^{2}
+C_m\|\psi(t)\|_{L_x^2}^{4m+2}\\
& &+\frac{\lambda_5}{\nu}\mbox{Re}\int_{\ttt}(\pt_x^{m-1}\psi)^2
\overline{\psi}^2dx+\frac{2\lambda_3+\lambda_4+2(m-1)\lambda_6}{4\nu}
\int_{\ttt}|\pt_x^{m-1}\psi|^2|\psi|^2dx,
\end{eqnarray*}
where $C_m$ is a sufficiently large constant depending only on 
$m$ so that $E_m(\psi)$ is positive. This is possible because 
of the following reason. The Gagliado-Nirenberg 
inequality (Lemma \ref{gn}) implies 
\begin{eqnarray*}
\frac{\lambda_5}{\nu}\mbox{Re}\int_{\ttt}(\pt_x^{m-1}\psi)^2
\overline{\psi}^2dx&+&\frac{2\lambda_3+\lambda_4+2(m-1)\lambda_6}{4\nu}
\int_{\ttt}|\pt_x^{m-1}\psi|^2|\psi|^2dx
\\
&\ge&-\frac12\|\pt_x^m\psi(t)\|_{L_x^2}^2
-\frac12\|\psi(t)\|_{L_x^2}^{2}
-D_m\|\psi(t)\|_{L_x^2}^{4m+2}
\end{eqnarray*}
with some positive constant 
$D_m$ depending only on $\nu, \lambda_3, \lambda_4, \lambda_5, 
\lambda_6$ and $m$. Hence we obtain
\begin{eqnarray*}
[E_m(\psi)](t)\ge\frac12\|\pt_x^m\psi(t)\|_{L_x^2}^2+
\frac12\|\psi(t)\|_{L_x^2}^2
+(C_m-D_m)\|\psi(t)\|_{L_x^2}^{4m+2}.
\end{eqnarray*}
Choosing $C_m$ so large that $C_m>D_m$, we have 
$[E_m(\psi)](t)>0$. We notice that
\begin{eqnarray}
\frac12\|\psi(t)\|_{H_x^m}^2\le[E_m(\psi)](t)\le C(\|\psi(t)\|_{L_x^2}^{4m}+1)
\|\psi(t)\|_{H_x^m}^2.\label{g4}
\end{eqnarray}

\begin{lemma}\label{mo} 
Let $\psi_{\epsilon}\in C([0,T_{\epsilon}),H^m(\ttt))$ be a solution to (\ref{3.1}). 
Then, there exists positive constants $C$ and $T=T(\|\phi\|_{H_x^m})$ 
which are 
independent of $\epsilon$ such that 
\begin{eqnarray*}
\|\psi_{\epsilon}(t)\|_{H_x^m}\le C(T,\|\phi\|_{L_x^2})
\|\phi\|_{H_x^m},
\end{eqnarray*}
for any $t\in[0,T)$.
\end{lemma}

\begin{proof} 
We first evaluate $[E_m(\psi)](t)$. 
Applying the $m$-th derivative the both sides of (\ref{3.1}), 
taking the inner product of the resultant equation with $\pt_x^m\psi$, and 
adding the complex conjugation of the produce, we obtain
\begin{eqnarray}
\lefteqn{\frac{d}{dt}\|\pt_x^m\psi_{\epsilon}(t)\|_{L_x^2}^2
+2\epsilon\|\pt_x^{m+2}\psi_{\epsilon}(t)\|_{L_x^2}^2}
\label{aa1}\\
&=&2\mbox{Im}\int_{\ttt}
\pt_x^m{{\mathcal N}}(\psi_{\epsilon},
\overline{\psi}_{\epsilon},
\ldots,\pt_x^2\psi_{\epsilon},\pt_x^2\overline{\psi}_{\epsilon})
\pt_x^m\overline{\psi}_{\epsilon}dx.
\nonumber
\end{eqnarray}
Using the Leibniz rule, we obtain
\begin{eqnarray}
\lefteqn{\pt_x^m{{\mathcal N}}(\psi_{\epsilon},
\overline{\psi}_{\epsilon},
\ldots,\pt_x^2\psi_{\epsilon},\pt_x^2\overline{\psi}_{\epsilon})}
\label{aa2}\\
&=&\{(2\lambda_3+m\lambda_6)\overline{\psi}_{\epsilon}\pt_x\psi_{\epsilon}
+(\lambda_4+m\lambda_6)\psi_{\epsilon}\pt_x\overline{\psi}_{\epsilon}\}\pt_x^{m+1}\psi_{\epsilon}
\nonumber\\
& &+(\lambda_4+2m\lambda_5)\psi_{\epsilon}\pt_x\psi_{\epsilon}\pt_x^{m+1}\overline{\psi}_{\epsilon}
+\lambda_6|\psi_{\epsilon}|^2\pt_x^{m+2}\psi_{\epsilon}+\lambda_5\psi_{\epsilon}^2
\pt_x^{m+2}\overline{\psi}_{\epsilon}
\nonumber\\
& &+P_1(\psi_{\epsilon},\overline{\psi}_{\epsilon},
\ldots,\pt_x^m\psi_{\epsilon},\pt_x^m\overline{\psi}_{\epsilon}),
\nonumber
\end{eqnarray}
where $P_1$ is a linear combination of the cubic terms 
$\pt_x^{j_1}\psi_{\epsilon}\pt_x^{j_2}\overline{\psi}_{\epsilon}\pt_x^{j_3}\psi_{\epsilon}$
with $j_1+j_2+j_3=m$, 
the cubic terms 
$\pt_x^{j_1}\psi_{\epsilon}\pt_x^{j_2}\overline{\psi}_{\epsilon}\pt_x^{j_3}\psi_{\epsilon}$ 
with $j_1+j_2+j_3=m+2$ and $j_1+j_2+j_3\le m$, and 
the quintic terms  
$\pt_x^{j_1}\psi_{\epsilon}\pt_x^{j_2}\overline{\psi}_{\epsilon}
\pt_x^{j_3}\psi_{\epsilon}
\pt_x^{j_4}\overline{\psi}_{\epsilon}\pt_x^{j_5}\psi_{\epsilon}
$
with $j_1+j_2+j_3+j_4+j_5=m$.
Hence the H\"{o}lder and Gagliardo-Nirenberg (Lemma \ref{gn})  inequalities imply
\begin{eqnarray}
\lefteqn{\|P_1\|_{L_x^2}}
\label{aa30}\\
&\le&C(\|\psi_{\epsilon}\|_{L_x^2}^{(2m-1)/m}
\|\psi_{\epsilon}\|_{H_x^m}^{(m+1)/m}
+\|\psi_{\epsilon}\|_{L_x^2}^{(4m-2)/m}
\|\psi_{\epsilon}\|_{H_x^m}^{(m+2)/m}\nonumber\\
& &+\|\psi_{\epsilon}\|_{L_x^2}^{(2m-3)/m}
\|\psi_{\epsilon}\|_{H_x^m}^{(m+3)/m})\nonumber\\
&\le&C[E_m(\psi_{\epsilon})](t)^{3/2}.
\nonumber
\end{eqnarray} 
In the last inequality we used the inequalities 
\begin{eqnarray*}
\|\psi\|_{L_x^2}&\le&[E_m(\psi)](t)^{\alpha}\qquad \mbox{for\ any}\ \ 
\frac{1}{4m+2}\le\alpha\le\frac12,\\
\|\psi\|_{H_x^m}&\le&[E_m(\psi)](t)^{1/2}.
\end{eqnarray*}
Substituting (\ref{aa2}) and (\ref{aa30}) into (\ref{aa1}), we have
\begin{eqnarray}
\lefteqn{\frac{d}{dt}\|\pt_x^m\psi_{\epsilon}(t)\|_{L_x^2}^2
+2\epsilon\|\pt_x^{m+2}\psi_{\epsilon}(t)\|_{L_x^2}^2}
\label{aa12}\\
&=&
2(2\lambda_3+m\lambda_6)\mbox{Im}\int_{\ttt}
\overline{\psi}_{\epsilon}\pt_x\psi_{\epsilon}
\pt_x^{m}\overline{\psi}_{\epsilon}\pt_x^{m+1}\psi_{\epsilon}dx\nonumber\\
& &+2(\lambda_4+m\lambda_6)\mbox{Im}\int_{\ttt}
\psi_{\epsilon}\pt_x\overline{\psi}_{\epsilon}
\pt_x^{m}\overline{\psi}_{\epsilon}\pt_x^{m+1}\psi_{\epsilon}dx\nonumber\\
& &+2(\lambda_4+2m\lambda_5)\mbox{Im}\int_{\ttt}
\psi_{\epsilon}\pt_x\psi_{\epsilon}\pt_x^{m}
\overline{\psi}_{\epsilon}\pt_x^{m+1}\overline{\psi}_{\epsilon}dx\nonumber\\
& &
+2\lambda_6\mbox{Im}\int_{\ttt}|\psi_{\epsilon}|^2
\pt_x^{m}\overline{\psi}_{\epsilon}\pt_x^{m+2}\psi_{\epsilon}dx
+2\lambda_5\mbox{Im}\int_{\ttt}\psi_{\epsilon}^2\pt_x^{m}\overline{\psi}_{\epsilon}
\pt_x^{m+2}\overline{\psi}_{\epsilon}dx\nonumber\\
& &+2\mbox{Im}\int_{\ttt}P_1(\psi_{\epsilon},\overline{\psi}_{\epsilon},
\ldots,\pt_x^m\psi_{\epsilon},\pt_x^m\overline{\psi}_{\epsilon})
\pt_x^m\overline{\psi}_{\epsilon}dx\nonumber\\
&\equiv&I_1+I_2+I_3+I_4+I_5+I_6.\nonumber
\end{eqnarray}
The inequality (\ref{aa30}) and the Schwarz inequality imply 
\begin{eqnarray*}
|I_6|\le\|P_1\|_{L_x^2}\|\pt_x^m\psi\|_{L_x^2}
\le C[E_m(\psi_{\epsilon})](t)^{2}.
\end{eqnarray*}
An integration by parts yields
\begin{eqnarray*}
|I_3|\le C[E_m(\psi_{\epsilon})](t)^{2}.
\end{eqnarray*}
We can express $I_2$, $I_4$ and $I_5$ in terms of $I_1$ by 
using an integration by parts:
\begin{eqnarray*}
I_2&=&2(\lambda_4+m\lambda_6)
\mbox{Im}\int_{\ttt}
\overline{\psi}_{\epsilon}\pt_x\psi_{\epsilon}
\pt_x^{m}\overline{\psi}_{\epsilon}\pt_x^{m+1}\psi_{\epsilon}dx\\
& &+R_1(\psi_{\epsilon},
\overline{\psi}_{\epsilon},
\ldots,\pt_x^m\psi_{\epsilon},\pt_x^m\overline{\psi}_{\epsilon}),\\
I_4&=&-4\lambda_6
\mbox{Im}\int_{\ttt}
\overline{\psi}_{\epsilon}\pt_x\psi_{\epsilon}
\pt_x^{m}\overline{\psi}_{\epsilon}\pt_x^{m+1}\psi_{\epsilon}dx
\\
& &+R_2(\psi_{\epsilon},
\overline{\psi}_{\epsilon},
\ldots,\pt_x^m\psi_{\epsilon},\pt_x^m\overline{\psi}_{\epsilon}),\\
I_5&=&-2\lambda_5
\mbox{Im}\int_{\ttt}
\psi_{\epsilon}^2(\pt_x^{m+1}\overline{\psi}_{\epsilon})^2dx
\\
& &+R_3(\psi_{\epsilon},\overline{\psi}_{\epsilon},
\ldots,\pt_x^m\psi_{\epsilon},\pt_x^m\overline{\psi}_{\epsilon}),
\end{eqnarray*}
where $R_1, R_2$ and $R_3$ satisfy 
\begin{eqnarray*}
|R_1|+|R_2|+|R_3|\le C[E_m(\psi_{\epsilon})](t)^{2}.
\end{eqnarray*}
Substituting above equations into (\ref{aa12}), we have
\begin{eqnarray}
\lefteqn{\frac{d}{dt}\|\pt_x^m\psi_{\epsilon}(t)\|_{L_x^2}^2
+2\epsilon\|\pt_x^{m+2}\psi(t)\|_{L_x^2}^2}\label{aaa1}\\
&=&2\{2\lambda_3+\lambda_4+2(m-1)\lambda_6\}
\mbox{Im}\int_{\ttt}\overline{\psi}_{\epsilon}
\pt_x\psi_{\epsilon}\cdot\pt_x^m\overline{\psi}_{\epsilon}
\pt_x^{m+1}\psi_{\epsilon}dx
\nonumber\\
& &-2\lambda_5\mbox{Im}
\int_{\ttt}\psi_{\epsilon}^2(\pt_x^{m+1}\overline{\psi}_{\epsilon})^2dx
+R_4(\psi_{\epsilon},
\overline{\psi}_{\epsilon},
\ldots,\pt_x^m\psi_{\epsilon},\pt_x^m\overline{\psi}_{\epsilon}),
\nonumber
\end{eqnarray}
where $R_4$ satisfies  
\begin{eqnarray*}
|R_4|\le C[E_m(\psi_{\epsilon})](t)^{2}.
\end{eqnarray*}
On the other hand, from the equation (\ref{4NLS}), we have 
\begin{eqnarray}
\lefteqn{\frac{d}{dt}\mbox{Re}\int_{\ttt}(\pt_x^{m-1}\psi_{\epsilon})^2
\overline{\psi}_{\epsilon}^2dx}
\label{o1}\\
&=&2\mbox{Re}\int_{\ttt}\pt_x^{m-1}\psi_{\epsilon}
\pt_t\pt_x^{m-1}\psi_{\epsilon}\cdot\overline{\psi}_{\epsilon}^2dx
+2\mbox{Re}\int_{\ttt}(\pt_x^{m-1}\psi_{\epsilon})^2
\overline{\psi}_{\epsilon}\pt_t\overline{\psi}_{\epsilon}dx
\nonumber\\
&=&-2\epsilon\mbox{Re}\int_{\ttt}\pt_x^{m-1}\psi_{\epsilon}
\pt_x^{m+3}\psi_{\epsilon}\cdot\overline{\psi}_{\epsilon}^2dx
-2\mbox{Im}\int_{\ttt}\pt_x^{m-1}\psi_{\epsilon}
\pt_x^{m+1}\psi_{\epsilon}\cdot\overline{\psi}_{\epsilon}^2dx
\nonumber\\
& &-2\nu\mbox{Im}\int_{\ttt}\pt_x^{m-1}\psi_{\epsilon}
\pt_x^{m+3}\psi_{\epsilon}\cdot\overline{\psi}_{\epsilon}^2dx
\nonumber\\
& &
+2\mbox{Im}\int_{\ttt}\pt_x^{m-1}\psi_{\epsilon}
\pt_x^{m-1}\mathcal{N}(\psi_{\epsilon},
\overline{\psi}_{\epsilon},
\ldots,\pt_x^2\psi_{\epsilon},\pt_x^2\overline{\psi}_{\epsilon})
\cdot\overline{\psi}_{\epsilon}^2dx
\nonumber\\
& &-2\epsilon\mbox{Re}\int_{\ttt}(\pt_x^{m-1}\psi_{\epsilon})^2
\overline{\psi}_{\epsilon}\pt_x^4\overline{\psi}_{\epsilon}dx
+2\mbox{Im}\int_{\ttt}(\pt_x^{m-1}\psi_{\epsilon})^2
\overline{\psi}_{\epsilon}\cdot\pt_x^2\overline{\psi}_{\epsilon}dx
\nonumber\\
& &
+2\nu\mbox{Im}\int_{\ttt}(\pt_x^{m-1}\psi_{\epsilon})^2
\overline{\psi}_{\epsilon}\cdot\pt_x^4\overline{\psi}_{\epsilon}dx
\nonumber\\
& &
-2\mbox{Im}\int_{\ttt}(\pt_x^{m-1}\psi_{\epsilon})^2
\overline{\psi}_{\epsilon}\cdot\overline{\mathcal{N}}
(\psi_{\epsilon},
\overline{\psi}_{\epsilon},
\ldots,\pt_x^2\psi_{\epsilon},\pt_x^2\overline{\psi}_{\epsilon})dx
\nonumber\\
&\equiv&I_7+I_8+I_9+I_{10}+I_{11}+I_{12}+I_{13}+I_{14}.
\nonumber
\end{eqnarray}
An integration by parts yields
\begin{eqnarray*}
I_7
&=&-2\epsilon\mbox{Re}\int_{\ttt}(\pt_x^{m+1}\psi_{\epsilon})^2\overline{\psi}_{\epsilon}^2dx
\\
& &+R_5(\psi_{\epsilon},
\overline{\psi}_{\epsilon},
\cdots,\pt_x^m\psi_{\epsilon},\pt_x^m\overline{\psi}_{\epsilon}),\\
I_9
&=&
-2\nu\mbox{Im}\int_{\ttt}(\pt_x^{m+1}\psi_{\epsilon})^2
\overline{\psi}_{\epsilon}^2dx
\\
& &+R_6(\psi_{\epsilon},
\overline{\psi}_{\epsilon},
\cdots,\pt_x^m\psi_{\epsilon},\pt_x^m\overline{\psi}_{\epsilon}),
\end{eqnarray*}
where $R_5$ and $R_6$ satisfy
\begin{eqnarray*}
|R_5|+|R_6|&\le&C\|\psi_{\epsilon}\|_{L_x^2}^{(2m-3)/m}
\|\psi_{\epsilon}\|_{H_x^m}^{(2m+3)/m}\\
&\le&C[E_m(\psi_{\epsilon})](t)^2.
\end{eqnarray*}
Integrating by parts, we also obtain 
\begin{eqnarray*}
|I_8|+|I_{12}|&\le&C\|\psi_{\epsilon}\|_{L_x^2}^{(2m-1)/m}
\|\psi_{\epsilon}\|_{H_x^m}^{(2m+1)/m}\\
&\le&C[E_m(\psi_{\epsilon})](t)^2,\\
|I_{11}|+|I_{13}|&\le&
C\|\psi_{\epsilon}\|_{L_x^2}^{(2m-3)/m}
\|\psi_{\epsilon}\|_{H_x^m}^{(2m+3)/m}\\
&\le&C[E_m(\psi_{\epsilon})](t)^2,\\
|I_{10}|+|I_{14}|&\le&C(
\|\psi_{\epsilon}\|_{L_x^2}^4\|\psi_{\epsilon}\|_{H_x^m}^2
+\|\psi_{\epsilon}\|_{L_x^2}^{(6m-1)/m}\|\psi_{\epsilon}\|_{H_x^m}^{(2m+1)/m}\\
& &+\|\psi_{\epsilon}\|_{L_x^2}^{(4m-2)/m}\|\psi_{\epsilon}\|_{H_x^m}^{(2m+2)/m})\\
&\le&C[E_m(\psi_{\epsilon})](t)^2.\\
\end{eqnarray*}
Substituting above equations into (\ref{o1}), we have
\begin{eqnarray}
\lefteqn{\frac{d}{dt}\mbox{Re}
\int_{\ttt}(\pt_x^{m-1}\psi_{\epsilon})^2\overline{\psi}_{\epsilon}^2dx
}\label{aa11}\\
&=&-2\nu\mbox{Im}\int_{\ttt}(\pt_x^{m+1}\psi_{\epsilon})^2\overline{\psi}_{\epsilon}^2dx
+2\epsilon\mbox{Re}\int_{\ttt}(\pt_x^{m+1}\psi_{\epsilon})^2\overline{\psi}_{\epsilon}^2dx
\nonumber\\
& &
+R_7(\psi_{\epsilon},
\overline{\psi}_{\epsilon},
\ldots,\pt_x^m\psi_{\epsilon},\pt_x^m\overline{\psi}_{\epsilon}),
\nonumber
\end{eqnarray}
where $R_7$ satisfies
\begin{eqnarray*}
|R_7|\le C[E_m(\psi_{\epsilon})](t)^2.
\end{eqnarray*}
By an argument similar to (\ref{aa11}), we obtain
\begin{eqnarray}
\lefteqn{
\frac{d}{dt}\int_{\ttt}|\pt_x^{m-1}\psi_{\epsilon}|^2|\psi_{\epsilon}|^2dx
}\label{aaa2}\\
&=&-8\nu\mbox{Im}\int_{\ttt}\overline{\psi}_{\epsilon}\pt_x\psi_{\epsilon}\cdot
\pt_x^m\overline{\psi}_{\epsilon}\pt_x^{m+1}\psi_{\epsilon}dx
-2\epsilon\int_{\ttt}|\pt_x^{m+1}\psi_{\epsilon}|^2|\psi_{\epsilon}|^2dx
\nonumber\\
& &+
R_8(\psi_{\epsilon},
\overline{\psi}_{\epsilon},
\ldots,\pt_x^m\psi_{\epsilon},\pt_x^m\overline{\psi}_{\epsilon}),
\nonumber
\end{eqnarray}
where $R_8$ satisfies
\begin{eqnarray*}
|R_8|\le C[E_m(\psi_{\epsilon})](t)^2.
\end{eqnarray*}
Finally, we obtain 
\begin{eqnarray}
\lefteqn{\frac{d}{dt}\|\psi_{\epsilon}(t)\|_{L_x^2}^{4m+2}
+2(2m+1)\epsilon\|\psi_{\epsilon}(t)\|_{L_x^2}^{4m}
\|\pt_x^2\psi_{\epsilon}(t)\|_{L_x^2}^2}
\label{aaa4}\\
& &\qquad\qquad\qquad
\qquad\qquad\qquad\qquad\qquad
\le C[E_m(\psi_{\epsilon})](t)^2.\nonumber
\end{eqnarray}
Collecting (\ref{o1}), (\ref{aa11}), (\ref{aaa2}), and (\ref{aaa4}),
we have
\begin{eqnarray}
\lefteqn{\frac{d}{dt}E_m(t)
+2\epsilon\|\pt_x^{m+2}\psi_{\epsilon}(t)\|_{L_x^2}^2
+2(2m+1)\epsilon\|\psi_{\epsilon}(t)\|_{L_x^2}^{4m}\|\pt_x^2\psi_{\epsilon}(t)\|_{L_x^2}^2}
\label{l1}\\
&=&
-\frac{2\lambda_5}{\nu}\epsilon\mbox{Re}\int_{\ttt}(\pt_x^{m+1}\psi_{\epsilon})^2
\overline{\psi}_{\epsilon}^2dx
\nonumber\\
& &-\frac{2\lambda_3+\lambda_4+2(m-1)\lambda_6}{2\nu}
\epsilon\int_{\ttt}|\pt_x^{m+1}\psi_{\epsilon}|^2|\psi_{\epsilon}|^2dx.
\nonumber\\
& &+
R_{9}(\psi_{\epsilon},\overline{\psi_{\epsilon}},
\ldots,\pt_x^m\psi_{\epsilon},\pt_x^m\overline{\psi}_{\epsilon}),
\nonumber
\end{eqnarray}
where $R_9$ satisfies
\begin{eqnarray*}
|R_9|\le C[E_m(\psi_{\epsilon})](t)^2.
\end{eqnarray*}
Since the sum of the first and second terms in the right hand side of (\ref{l1})
are bounded by 
$\displaystyle{\epsilon\|\pt_x^{m+2}\psi_{\epsilon}
\|_{L_x^2}^2+C[E_m(\psi_{\epsilon})](t)^{2}}$,
we obtain
\begin{eqnarray*}
\frac{d}{dt}[E_m(\psi_{\epsilon})](t)
&+&\epsilon\|\pt_x^{m+2}\psi_{\epsilon}(t)\|_{L_x^2}^2
+2(2m+1)\epsilon
\|\psi_{\epsilon}(t)\|_{L_x^2}^{4m}\|\pt_x^2\psi_{\epsilon}(t)\|_{L_x^2}^2
\\
&\le&C[E_m(\psi_{\epsilon})](t)^2.
\end{eqnarray*}
Therefore
\begin{eqnarray*}
\frac{d}{dt}[E_m(\psi_{\epsilon})](t)\le C[E_m(\psi_{\epsilon})](t)^2.
\end{eqnarray*}
We note that the constant $C$ is independent of $\epsilon\in(0,1]$.
From the above inequality we have 
\begin{eqnarray*}
[E_m(\psi_{\epsilon})](t)\le\frac{[E_m(\psi_{\epsilon})](0)}{
1-Ct[E_m(\psi_{\epsilon})](0)}.
\end{eqnarray*}
for $\displaystyle{0\le t<\min\{T_{\epsilon},C^{-1}[E_m(\psi_{\epsilon})](0)^{-1}\}}$. 
Combing this inequality, 
$\|\phi_{\epsilon}\|_{H^m}\le\|\phi\|_{H^m}$ for any $\epsilon\in(0,1]$ and 
(\ref{g4}), 
we see
\begin{eqnarray*}
\|\psi_{\epsilon}\|_{H_x^m}^2\le\frac{C(\|\phi\|_{L_x^2}^{4m}+1)\|\phi\|_{H_x^m}^2}{
1-Ct(\|\phi\|_{L_x^2}^{4m}+1)\|\phi\|_{H_x^m}^2}.
\end{eqnarray*}
for $\displaystyle{0\le t<\min\{T_{\epsilon},C^{-1}(\|\phi\|_{L_x^2}^{4m}
+1)^{-1}
\|\phi\|_{H_x^m}^{-2}\}}$. 
Let $T\equiv(2C)^{-1}(\|\phi\|_{L_x^2}^{4m}
+1)^{-1}\|\phi\|_{H_x^m}^{-2}$. 
Then for any $0<t<\min\{T_{\epsilon},
T\}$, we have
\begin{eqnarray*}
\|\psi_{\epsilon}\|_{H_x^m}^2\le2C(\|\phi\|_{L_x^2}^{4m}
+1)\|\phi\|_{H_x^m}^2.
\end{eqnarray*}
If $T_{\epsilon}<T$, 
we can apply Lemma \ref{l3.1} to extend the solution in the same class to 
the interval $[0,T)$. 
Therefore we obtain the desired result. 
\end{proof}

Using Lemma \ref{mo} we obtain the existence 
of the solution to (\ref{4NLS}):

\begin{lemma}\label{ell} 
Let $m\ge4$ be 
an integer. 
For any $\phi\in H^m(\ttt)$, there exists a time 
$T=T(\|\phi\|_{H^m})>0$ and a solution 
$\psi$ of (\ref{4NLS}) satisfying
\begin{eqnarray*}
\psi\in L^{\infty}([0,T);H^m(\rre)).
\end{eqnarray*} 
\end{lemma}

\begin{proof} Let 
$\phi\in H^m(\ttt)$ and let $\{\phi_{\epsilon}\}_{\epsilon}\subset 
H^{\infty}(\ttt)$ be a Bona-Smith approximation of $\phi$. 
Then by Lemma \ref{l3.1} there exists 
a unique solution $\psi_{\epsilon}\in C([0,T_{\epsilon});H^m(\ttt))$ 
to (\ref{3.1}). Lemma \ref{mo} yields that there exists $T=T(\|\phi\|_{H_x^m})>0$ 
which is independent of $\epsilon$ 
such that $\{\psi_{\epsilon}\}_{\epsilon}$ 
is uniformly bounded in $L^{\infty}(0,T;H^m(\ttt))$ with respect to $\epsilon\in(0,1]$. 
By a standard limiting argument, it is 
inferred that a subsequence of $\psi^{\epsilon}$ convergence in $L^{\infty}(0,T;H^m(\ttt))$ 
weak$^{\ast}$ to a solution $\psi$ of (\ref{4NLS}) such that $\psi^{\epsilon}\in 
L^{\infty}(0,T;H^m(\ttt))$. We omit the detail. 
\end{proof}

\section{Proof of Theorem 
\ref{well-posedness}} \label{sec:unique}

In the preceding sections, we proved 
the existence of the solution to (\ref{4NLS}). 
In this section, we complete the 
proof of Theorem \ref{well-posedness} by 
showing the following three assertions

\vskip2mm

(i) uniqueness of the solution

(ii) persistent properties of the solution 

(iii) continuous dependence of the solution upon initial data

\subsection{Uniqueness}

Let $\psi_1$ and $\psi_2$ be two solutions to (\ref{4NLS}) 
with same initial data satisfying 
$\sup_{t\in[0,T)}\|\psi_j(t)\|_{H_x^m}
<\infty$, $j=1,2$. 
We shall show that $\psi_1\equiv\psi_2$ for $t\in[0,T)$. 
To prove this, it suffices to show that 
$\psi=\psi_2-\psi_1$ satisfies $\|\psi(t)\|_{H_x^1}\equiv0$ 
because this identity and $\psi(0)\equiv0$ implies $\psi\equiv0$. 
The reason we prove $\|\psi(t)\|_{H_x^1}\equiv0$ instead of 
driving $\|\psi(t)\|_{L_x^2}\equiv0$ is that the corresponding 
modified energy for $L^2$ involves the anti-derivatives 
of $\psi$. 

The standard energy estimate yields 
\begin{eqnarray}
\lefteqn{\frac{d}{dt}\|\psi(t)\|_{L_x^2}^2}
\label{y0}\\
&=&
2\mbox{Im}
\int_{\ttt}
\{\mathcal{N}(\psi+\psi_1,\overline{\psi}+\overline{\psi}_1
,\ldots,\pt_x^2\psi+\pt_x^2\psi_1,
\pt_x^2\overline{\psi}+\pt_x^2\overline{\psi}_1)
\nonumber\\
& &\qquad\qquad\qquad\qquad
\qquad\qquad
-\mathcal{N}(\psi_1,\overline{\psi}_1
,\ldots,\pt_x^2\psi_1,
\pt_x^2\overline{\psi}_1)\}\overline{\psi}dx
\nonumber\\
&\le&C(
(\|\psi_1\|_{H_x^2}^2+\|\psi_1\|_{H_x^2}^4
+\|\psi_2\|_{H_x^2}^2+\|\psi_2\|_{H_x^2}^4)
\|\psi\|_{H_x^1}^2,\nonumber\\
\lefteqn{\frac{d}{dt}\|\pt_x\psi(t)\|_{L_x^2}}
\label{y1}\\
&=&
2\mbox{Im}
\int_{\ttt}
\pt_x\{\mathcal{N}(\psi+\psi_1,\overline{\psi}+\overline{\psi}_1
,\ldots,\pt_x^2\psi+\pt_x^2\psi_1,
\pt_x^2\overline{\psi}+\pt_x^2\overline{\psi}_1)
\nonumber\\
& &\qquad\qquad\qquad\qquad
\qquad\qquad
-\mathcal{N}(\psi_1,\overline{\psi}_1,
\ldots,\pt_x^2\psi_1,
\pt_x^2\overline{\psi}_1)\}\pt_x\overline{\psi}dx
\nonumber\\
&=&2(2\lambda_3+\lambda_4)
\mbox{Im}\int_{\ttt}\pt_x\psi_1\overline{\psi}_1
\cdot\pt_x^2\psi\pt_x\overline{\psi}dx
-2\lambda_5\mbox{Im}\int_{\ttt}\psi_1^2(\pt_x^2\overline{\psi})^2dx
\nonumber\\
& &+
R_{10}(\psi_1,\overline{\psi}_1,\ldots,\pt_x^3\psi_1,\pt_x^3\overline{\psi}_1,
\psi_2,\overline{\psi}_2,\ldots,\pt_x^3\psi_2,\pt_x^3\overline{\psi}_2),
\nonumber
\end{eqnarray}
where $R_{10}$ satisfies
\begin{eqnarray*}
|R_{10}|\le C(\|\psi_1\|_{H_x^3}^2+\|\psi_1\|_{H_x^3}^4
+\|\psi_2\|_{H_x^3}^2+\|\psi_2\|_{H_x^3}^4)
\|\psi\|_{H_x^1}^2.
\end{eqnarray*}
On the other hand, a direct calculation yields
\begin{eqnarray}
\lefteqn{\frac{2\lambda_3+\lambda_4}{4\nu}
\frac{d}{dt}\int_{\ttt}|\psi_1|^2|\psi|^2dx}
\label{y2}\\
&=&-2(2\lambda_3+\lambda_4)
\mbox{Im}\int_{\ttt}\pt_x\psi_1\overline{\psi}_1
\cdot\pt_x^2\psi\pt_x\overline{\psi}dx
\nonumber\\
& &+R_{11}(\psi_1,\overline{\psi}_1,\ldots,
\pt_x^3\psi_1,\pt_x^3\overline{\psi}_1,
\psi_2,\overline{\psi}_2,\ldots,\pt_x^3\psi_2,\pt_x^3
\overline{\psi}_2),
\nonumber
\end{eqnarray}
\begin{eqnarray}
\lefteqn{\frac{\lambda_5}{\nu}
\frac{d}{dt}\mbox{Re}\int_{\ttt}\psi_1^2(\overline{\psi})^2dx}
\label{w3}\\
&=&2\lambda_5\mbox{Im}\int_{\ttt}\psi_1^2
(\pt_x^2\overline{\psi})^2dx
\nonumber\\
& &+R_{12}(\psi_1,\overline{\psi}_1,\ldots,\pt_x^3\psi_1,
\pt_x^3\overline{\psi}_1,
\psi_2,\overline{\psi}_2,\ldots,
\pt_x^3\psi_2,\pt_x^3\overline{\psi}_2),
\nonumber
\end{eqnarray}
where $R_{11}$ and $R_{12}$ satisfy 
\begin{eqnarray*}
|R_{11}|+|R_{12}|\le C(\|\psi_1\|_{H_x^3}^2+\|\psi_1\|_{H_x^3}^6+
\|\psi_2\|_{H_x^3}^2+\|\psi_2\|_{H_x^3}^6)
\|\psi\|_{H_x^1}^2.
\end{eqnarray*}
Here we set
\begin{eqnarray*}
[\tilde{E}_1(\psi)](t)
&=&\|\pt_x\psi(t)\|_{L_x^2}^2
+\tilde{C}_1\|\psi(t)\|_{L_x^2}^2
\\
& &+\frac{2\lambda_3+\lambda_4}{4\nu}\int_{\ttt}
|\psi_1|^2|\psi|^2dx
+
\frac{\lambda_5}{\nu}\mbox{Re}\int_{\ttt}
\psi_1^2(\overline{\psi})^2dx,
\end{eqnarray*}
where $\tilde{C}_1$ is a sufficiently large constant depending 
only on 
$m$, $\sup_{t\in[0,T)}\|\psi_1(t)\|_{H_x^1}$ and 
 $\sup_{t\in[0,T)}\|\psi_2(t)\|_{H_x^1}$ 
so that $\tilde{E}_1(\psi)$ is positive. Then, from 
(\ref{y0}), (\ref{y1}), (\ref{y2}) and (\ref{w3}), we obtain
\begin{eqnarray*}
\lefteqn{\frac{d}{dt}[\tilde{E}_1(\psi)](t)}\\
&\le&
C(\|\psi_1\|_{H_x^3}^2+\|\psi_1\|_{H_x^3}^6+\|\psi_2\|_{H_x^3}^2+\|\psi_2\|_{H_x^3}^6)
\|\psi\|_{L_x^2}^2\\
&\le&C[\tilde{E}_1(\psi)](t).
\end{eqnarray*}
Hence Gronwall's lemma yields 
\begin{eqnarray}
[\tilde{E}_1(\psi)](t)\le
[\tilde{E}_1(\psi)](0)e^{ct}.\label{y4}
\end{eqnarray}
Since $[\tilde{E}_m(\psi)](0)=0$, Gronwall's lemma yields
$[\tilde{E}_1(\psi)](t)\equiv0$. Combination of this identity and 
the equality $0\le\|\psi(t)\|_{H_x^1}
\le[\tilde{E}_1(\psi)](t)$ implies $\psi\equiv0$, 
which completes the proof of the uniqueness. 


\subsection{Persistence of solution} \label{sec:contu}

To prove the persistent property of the solution to (\ref{4NLS}) 
which is obtained by Lemma \ref{ell}, 
we employ the Bona-Smith approximation. 
We denote $\phi_{\epsilon}$ the Bona-Smith approximation 
of $\phi$. 

\begin{lemma}\label{m2} 
Let $\phi\in H^m(\ttt)$ with $m\ge3$ and 
let $\psi_{\alpha}$, $\psi_{\epsilon}$ denote 
the solution to (\ref{4NLS}) corresponding 
to the initial data $\phi_{\alpha}$ and $\phi_{\epsilon}$, 
respectively. Then there exists $C=C(T,\|\phi\|_{H_x^m})>0$ 
such that for $0\le\alpha<\epsilon\le1$, 
\begin{eqnarray}
\lefteqn{\sup_{t\in[0,T)}
\|\psi_{\alpha}(t)-\psi_{\epsilon}(t)\|_{H_x^m}}
\label{ww}\\
&\le&C(\epsilon^{m-3}+\|\phi-\phi_{\alpha}\|_{H_x^m}
+\|\phi-\phi_{\epsilon}\|_{H_x^m}).
\nonumber
\end{eqnarray}
\end{lemma}

\begin{proof} 
We put $\psi=\psi_{\alpha}-\psi_{\epsilon}$. 
We first evaluate $\|\psi(t)\|_{H_x^1}$. 
Replacing $\psi=\psi_1-\psi_2$ by 
$\psi=\psi_{\alpha}-\psi_{\epsilon}$ in (\ref{y4}), 
we have
\begin{eqnarray}
[\tilde{E}_1(\psi)](t)\le
[\tilde{E}_1(\psi)](0)e^{Ct}, 
\label{x1}
\end{eqnarray}
for $t\in[0,T)$, where 
\begin{eqnarray*}
[\tilde{E}_1(\psi)](t)
&=&\|\pt_x\psi(t)\|_{L_x^2}^2
+\tilde{C}_1\|\psi(t)\|_{L_x^2}^2
\\
& &+\frac{2\lambda_3+\lambda_4}{4\nu}\int_{\ttt}
|\psi_{\alpha}|^2|\psi|^2dx
+
\frac{\lambda_5}{\nu}\mbox{Re}\int_{\ttt}
\psi_{\alpha}^2(\overline{\psi})^2dx,
\end{eqnarray*}
$\tilde{C}_1$ is a sufficiently large constant depending only on 
$m$, $\sup_{t\in[0,T)}\|\psi(t)\|_{H_x^m}$ 
\footnote{By the inequality $\|\phi_{\epsilon}\|_{H_x^m}
\le\|\phi\|_{H^m}$ and Lemma \ref{ell}, 
we can choose $\tilde{C}_1$ independently of $\alpha$.}
so that $\tilde{E}_1(\psi)$ is positive and 
$C$ in (\ref{x1}) depends only on 
$\sup_{t\in[0,T)}\|\psi(t)\|_{H_x^3}$.  
Since 
\begin{eqnarray}
\|\psi(t)\|_{H_x^1}^2&\le&[\tilde{E}_1(\psi)](t),
\label{x2}\\
\mathop{[}\tilde{E}_1(\psi)](0)
&\le&C\|\phi_{\alpha}-\phi_{\epsilon}\|_{H_x^1}^2
\le C(\|\phi-\phi_{\alpha}\|_{H_x^1}^2
+\|\phi-\phi_{\epsilon}\|_{H_x^1}^2)\label{x3}\\
&\le&C(\alpha^{2(m-1)}+\epsilon^{2(m-1)})
\le C\epsilon^{2(m-1)},\nonumber
\end{eqnarray}
the inequalities (\ref{x1}), (\ref{x2}) and (\ref{x3}) lead to 
the inequality  
\begin{eqnarray}
\|\psi(t)\|_{H_x^1}\le C\epsilon^{m-1}.
\label{r1}
\end{eqnarray}
Next, we evaluate $\|\psi(t)\|_{H_x^m}$. 
By an argument similar to (\ref{aa1}),
\begin{eqnarray}
\lefteqn{\frac{d}{dt}\|\pt_x^m\psi(t)\|_{L_x^2}^2}
\label{bb1}\\
&=&2\mbox{Im}\int_{\ttt}
\pt_x^m\{{{\mathcal N}}(\psi+\psi_{\epsilon},
\overline{\psi}+\overline{\psi}_{\epsilon},
\ldots,\pt_x^2\psi+\pt_x^2\psi_{\epsilon},\pt_x^2\overline{\psi}+\pt_x^2\overline{\psi}_{\epsilon})
\nonumber\\
& &\qquad\qquad\qquad\qquad\qquad\qquad
-{{\mathcal N}}(\psi_{\epsilon},
\overline{\psi}_{\epsilon},\ldots,
\pt_x^2\psi_{\epsilon},
\pt_x^2\overline{\psi}_{\epsilon})\}\pt_x^m\overline{\psi}dx.
\nonumber
\end{eqnarray}
Using the Leibniz rule, we have
\begin{eqnarray}
\lefteqn{\pt_x^m\{{{\mathcal N}}(\psi+\psi_{\epsilon},
\overline{\psi}+\overline{\psi}_{\epsilon},\ldots,
\pt_x^2\psi+\pt_x^2\psi_{\epsilon},
\pt_x^2\overline{\psi}+\pt_x^2\overline{\psi}_{\epsilon})}
\label{x6}\\
& &\qquad\qquad\qquad\qquad\qquad\qquad
-{{\mathcal N}}(\psi_{\epsilon},\overline{\psi}_{\epsilon}
\ldots,\pt_x^2\psi_{\epsilon},
\pt_x^2\overline{\psi}_{\epsilon})\}
\nonumber\\
&=&\{(2\lambda_3+m\lambda_6)\overline{\psi}_{\alpha}\pt_x\psi_{\alpha}
+(\lambda_4+m\lambda_6)\psi_{\alpha}\pt_x\overline{\psi}_{\alpha}\}\pt_x^{m+1}\psi
\nonumber\\
& &+(\lambda_4+2m\lambda_5)\psi_{\alpha}\pt_x\psi_{\alpha}\pt_x^{m+1}\overline{\psi}
\nonumber\\
& &+\lambda_6|\psi_{\alpha}|^2\pt_x^{m+2}\psi+\lambda_5\psi_{\alpha}^2\pt_x^{m+2}\overline{\psi}
\nonumber\\
& &+P_2(\psi_{\alpha},
\overline{\psi}_{\alpha},\ldots,
\pt_x^m\psi_{\alpha}, 
\pt_x^m\overline{\psi}_{\alpha},
\psi_{\epsilon},
\overline{\psi}_{\epsilon},\ldots,
\pt_x^{m+2}\psi_{\epsilon},
\pt_x^{m+2}\overline{\psi}_{\epsilon}),
\nonumber
\end{eqnarray}
and
\begin{eqnarray*}
\lefteqn{P_2(\psi_{\alpha},
\overline{\psi}_{\alpha},\ldots,
\pt_x^m\psi_{\alpha}, 
\pt_x^m\overline{\psi}_{\alpha},
\psi_{\epsilon},
\overline{\psi}_{\epsilon},\ldots,
\pt_x^{m+2}\psi_{\epsilon},
\pt_x^{m+2}\overline{\psi_{\epsilon}})}\\
&=&\{(2\lambda_3+m\lambda_6)
(\overline{\psi}_{\alpha}\pt_x\psi
+\pt_x\psi_{\epsilon}\overline{\psi})
+(\lambda_4+m\lambda_6)
\psi_{\alpha}\pt_x\overline{\psi}+\pt_x\overline{\psi}_{\epsilon}
\psi\}\pt_x^{m+1}\psi_{\epsilon}\\
& &+(\lambda_4+2m\lambda_5)
(\psi_{\alpha}\pt_x\psi+\pt_x\psi_{\epsilon}\psi)\pt_x^{m+1}
\overline{\psi}_{\epsilon}\\
& &+\lambda_6(\overline{\psi}_{\alpha}\psi+\psi_{\epsilon}\overline{\psi})\pt_x^{m+2}\psi_{\epsilon}
+\lambda_5(\psi_{\alpha}\psi+\psi_{\epsilon}\psi)\pt_x^{m+2}\overline{\psi}_{\epsilon}\\
& &+P_3(\psi_{\alpha},\ldots,\pt_x^m\overline{\psi}_{\alpha},
\psi_{\epsilon},\ldots,
\pt_x^m\overline{\psi}_{\epsilon}),
\end{eqnarray*}
where $P_3$ satisfies
\begin{eqnarray}
\|P_3\|_{L_x^2}&\le&C(\|\psi_{\alpha}\|_{H_x^{m}}^2
+\|\psi_{\alpha}\|_{H_x^m}^4
+\|\psi_{\epsilon}\|_{H_x^{m}}^2
+\|\psi_{\epsilon}\|_{H_x^m}^4)
\|\psi\|_{H_x^{m}}\label{x3}\\
&\le&C\|\psi\|_{H_x^{m}}.\nonumber
\end{eqnarray}
Combining the inequalities (\ref{r1}) and (\ref{x3}) and 
$\|\psi_{\epsilon}\|_{H_x^{m+2}}
\le C\epsilon^{-2}\|\psi\|_{H_x^m}$, 
we have 
\begin{eqnarray*}
\|P_2\|_{L_x^2}&\le&
C(\|\psi_{\alpha}\|_{H_x^1}+\|\psi_{\epsilon}\|_{H_x^1})
\|\psi_{\epsilon}\|_{H_x^{m+2}}
\|\psi\|_{H_x^1}+C\|\psi\|_{H_x^m}\\
&\le&C\epsilon^{m-3}+C\|\psi\|_{H_x^m}.
\end{eqnarray*}
The identity (\ref{x6}) and a standard energy estimate yield
\begin{eqnarray}
\lefteqn{\frac{d}{dt}\|\pt_x^{m}\psi(t)\|_{L_x^2}^2}
\label{y3}\\
&=&2\{2\lambda_3+\lambda_4+2(m-1)\lambda_6\}
\mbox{Im}\int_{\ttt}
\pt_x\psi_{\alpha}\overline{\psi}_{\alpha}
\pt_x^{m+1}\psi\pt_x^{m}\overline{\psi}dx
\nonumber\\
& &-2\lambda_5\mbox{Im}\int_{\ttt}\psi_{\alpha}^2(\pt_x^{m+1}\overline{\psi})^2dx
\nonumber\\
& &
+R_{13}(\psi_{\alpha},\overline{\psi}_{\alpha},\ldots,
\pt_x^m\psi_{\alpha}, 
\pt_x^m\overline{\psi}_{\alpha},
\psi_{\epsilon},
\overline{\psi}_{\epsilon},\ldots,
\pt_x^{m+2}\psi_{\epsilon},\pt_x^{m+2}\overline{\psi}_{\epsilon}),
\nonumber
\end{eqnarray}
where $R_{13}$ satisfies
\begin{eqnarray*}
|R_{13}|\le C\epsilon^{m-3}+C\|\psi\|_{H_x^m}^2.
\end{eqnarray*}
On the other hand, 
a direct calculation yields
\begin{eqnarray}
\lefteqn{\frac{2\lambda_3+\lambda_4+2(m-1)\lambda_6}{4\nu}
\frac{d}{dt}\int_{\ttt}|\psi_{\alpha}|^2|\pt_x^{m-1}\psi|^2dx}
\label{z2}\\
&=&-2\{2\lambda_3+\lambda_4+2(m-1)\lambda_6\}
\mbox{Im}\int_{\ttt}
\pt_x\psi_{\alpha}\overline{\psi}_{\alpha}\pt_x^{m+1}\psi\pt_x^{m}\overline{\psi}dx
\nonumber\\
& &+R_{14}(\psi_{\alpha},\overline{\psi}_{\alpha},\ldots,
\pt_x^m\psi_{\alpha}, 
\pt_x^m\overline{\psi}_{\alpha},
\psi_{\epsilon},
\overline{\psi}_{\epsilon},\ldots,
\pt_x^{m}\psi_{\epsilon},\pt_x^{m}\overline{\psi}_{\epsilon}),
\nonumber
\end{eqnarray}
and
\begin{eqnarray}
\lefteqn{\frac{\lambda_5}{\nu}\frac{d}{dt}\mbox{Re}\int_{\ttt}\psi_{\alpha}^2(\pt_x^{m-1}\overline{\psi})^2dx}
\label{z3}\\
&=&2\lambda_5\mbox{Im}\int_{\ttt}\psi_{\alpha}^2(\pt_x^{m+1}\overline{\psi})^2dx
\nonumber\\
& &
+R_{15}(\psi_{\alpha},\overline{\psi}_{\alpha},\ldots,
\pt_x^m\psi_{\alpha}, 
\pt_x^m\overline{\psi}_{\alpha},
\psi_{\epsilon},
\overline{\psi}_{\epsilon},\ldots,
\pt_x^{m}\psi_{\epsilon},\pt_x^{m}\overline{\psi}_{\epsilon}),
\nonumber
\end{eqnarray}
where $R_{14}$ and $R_{15}$ satisfy
\begin{eqnarray*}
|R_{14}|+|R_{15}|\le C\|\psi\|_{H_x^m}^2.
\end{eqnarray*}
From (\ref{y3}), (\ref{z2}) and (\ref{z3}), we obtain
\begin{eqnarray}
\frac{d}{dt}\{\|\pt_x^{m}\psi(t)\|_{L_x^2}^2
&+&\frac{2\lambda_3+\lambda_4+2(m-1)\lambda_6}{4\nu}
\int_{\ttt}|\psi_{\alpha}|^2|\pt_x^{m-1}\psi|^2dx
\label{x7}\\
&+&
\frac{\lambda_5}{\nu}\mbox{Re}\int_{\ttt}\psi_{\alpha}^2(\pt_x^{m-1}\overline{\psi})^2dx
\}
\le C\|\psi\|_{H_x^{m}}^2+C\epsilon^{m-3}.
\nonumber
\end{eqnarray}
Here we set
\begin{eqnarray*}
[\tilde{E}_m(\psi)](t)
&=&\int_{\ttt}|\pt_x^{m}\psi|^2dx
+\tilde{C}_m\int_{\ttt}|\psi|^2dx\\
& &+\int_{\ttt}|\psi_{\alpha}|^2|\pt_x^{m-1}\psi|^2dx
+\mbox{Re}\int_{\ttt}\psi_{\alpha}^2(\pt_x^{m-1}\overline{\psi})^2dx,
\end{eqnarray*}
where $\tilde{C}_m$ is a sufficiently large constant depending only on 
$m$, $\sup_{t\in[0,T)}\|\psi(t)\|_{H_x^m}$  
so that $\tilde{E}_m(\psi)$ is positive. Then 
the inequality (\ref{x7}) is expressed in terms of $\tilde{E}_m$:
\begin{eqnarray*}
\frac{d}{dt}[\tilde{E}_m(\psi)](t)\le
C[\tilde{E}_m(\psi)](t)+C\epsilon^{m-3}.
\end{eqnarray*}
Therefore Gronwall's lemma leads to the inequality 
\begin{eqnarray*}
[\tilde{E}_m(\psi)](t)\le C([\tilde{E}_m(\psi)](0)
+\epsilon^{m-3})e^{CT}.
\end{eqnarray*}
Therefore we have (\ref{ww}) which completes 
Lemma \ref{m2}. \end{proof}

\vskip2mm

Let us prove the persistent property of the solution to 
(\ref{4NLS}). Let $\phi\in H^m(\ttt)$ and 
$\{\phi_{\epsilon}\}_{\epsilon>0}\subset H^{\infty}(\ttt)$ 
be a Bona-Smith approximation of $\phi$. 
Lemma \ref{ell} yields there exists $T=T(\|\phi_{\epsilon}\|_{H_x^m})$ 
and a unique solution $\psi_{\epsilon}(t)\in L^{\infty}(0,T;H^{\infty}(\ttt))$ 
to (\ref{4NLS}). Since $\|\phi_{\epsilon}\|_{H^m}
\le \|\phi\|_{H^m}$, we can choose $T$ independently of $\epsilon$. 
By Lemma \ref{m2}, $\{\psi_{\epsilon}(t)\}_{\epsilon}$ 
is Cauchy sequence in $C(0,T;H^m(\ttt))$. Consequently  
we see that $\phi\in C(0,T;H^m(\ttt))$. This gauarantees 
the persistent property of the solution in Theorem \ref{well-posedness}. 


\subsection{Continuity of data-to-solution map} \label{sec:contu}

As the final step of the proof of Theorem \ref{well-posedness}, 
we prove that  
the data-to-solution map $S_t:H^m(\ttt)\to C([0,T);H^m(\ttt))$ 
$(\phi\mapsto\psi(t))$ associated to (\ref{4NLS}) is continuous. 
To this end, we shall prove 
the following: Let $\phi\in H^m(\ttt)$. For any $\eta>0$ 
there exists $\delta>0$ such that if $\tilde{\phi}\in H^m(\ttt)$ 
satisfies 
\begin{eqnarray*}
\|\phi-\tilde{\phi}\|_{H_x^m}<\delta,
\end{eqnarray*}
then 
\begin{eqnarray*}
\sup_{t\in[0,T)}\|S_t(\phi)-S_t(\tilde{\phi})\|_{H_x^m}<\eta.
\end{eqnarray*}
Let $\{\phi_{\epsilon}\}_{\epsilon>0}$ and 
$\{\tilde{\phi}_{\epsilon}\}_{\epsilon>0}$ be the Bona-Smith 
approximations of $\phi$ and $\tilde{\phi}$, respectively.
By the triangle inequality, we have
\begin{eqnarray}
\lefteqn{\sup_{t\in[0,T)}\|S_t(\phi)-S_t(\tilde{\phi})\|_{H_x^m}}
\label{k1}\\
&\le&
\sup_{t\in[0,T)}\|S_t(\phi)-S_t(\phi_{\epsilon})\|_{H_x^m}
+\sup_{t\in[0,T)}\|S_t(\phi_{\epsilon})-S_t(\tilde{\phi}_{\epsilon})\|_{H_x^m}
\nonumber\\
& &+
\sup_{t\in[0,T)}\|S_t(\tilde{\phi}_{\epsilon})
-S_t(\tilde{\phi})\|_{H_x^m}.
\nonumber
\end{eqnarray}
Letting $\alpha$ tend to $0$ in (\ref{ww}), we have
\begin{eqnarray}
\sup_{t\in[0,T)}\|S_t(\phi)-S_t(\phi_{\epsilon})\|_{H_x^m}
&\le&C(\epsilon^{m-3}+\|\phi-\phi_{\epsilon}\|_{H_x^m}),
\label{k2}\\
\sup_{t\in[0,T)}\|S_t(\tilde{\phi}_{\epsilon})-S_t(\tilde{\phi})\|_{H_x^m}
&\le&C(\epsilon^{m-3}+\|\tilde{\phi}-
\tilde{\phi}_{\epsilon}\|_{H_x^m}).
\label{k3}
\end{eqnarray}
By a similar argument as the derivation of (\ref{ww}), 
we obtain
\begin{eqnarray*}
\sup_{t\in[0,T)}\|S_t(\phi_{\epsilon})-S_t
(\tilde{\phi}_{\epsilon})\|_{H_x^m}
&\le&C(\epsilon^{m-3}+\|\phi_{\epsilon}-
\tilde{\phi}_{\epsilon}\|_{H_x^m}).
\end{eqnarray*}
Combining the above inequality with the triangle inequality 
\begin{eqnarray*}
\|\phi_{\epsilon}-
\tilde{\phi}_{\epsilon}\|_{H_x^m}
\le\|\phi_{\epsilon}-
\phi\|_{H_x^m}
+
\|\phi-
\tilde{\phi}\|_{H_x^m}
+
\|\tilde{\phi}-
\tilde{\phi}_{\epsilon}\|_{H_x^m},
\end{eqnarray*}
we have
\begin{eqnarray}
\lefteqn{\sup_{t\in[0,T)}\|S_t(\phi_{\epsilon})-S_t
(\tilde{\phi}_{\epsilon})\|_{H_x^m}}
\label{k4}\\
&\le&
C(\epsilon^{m-3}
+\|\phi_{\epsilon}-
\phi\|_{H_x^m}
+
\|\phi-
\tilde{\phi}\|_{H_x^m}
+
\|\tilde{\phi}-
\tilde{\phi}_{\epsilon}\|_{H_x^m}.
\nonumber
\end{eqnarray}
Substituting (\ref{k2}), (\ref{k3}) and (\ref{k4}) into 
(\ref{k1}), we obtain
\begin{eqnarray}
\lefteqn{\sup_{t\in[0,T)}\|S_t(\phi)-S_t(\tilde{\phi})\|_{H_x^m}}
\label{k5}\\
&\le& 
C(\epsilon^{m-3}
+\|\phi_{\epsilon}-\phi\|_{H_x^m}
+\|\phi-
\tilde{\phi}\|_{H_x^m}
+\|\tilde{\phi}-
\tilde{\phi}_{\epsilon}\|_{H_x^m}).
\nonumber
\end{eqnarray}
We first choose $\delta>0$ so 
that $C\delta<\eta/4$. 
Since $\phi_{\epsilon}\to\phi$ and 
$\tilde{\phi}_{\epsilon}\to\tilde{\phi}$ 
in $H^m$ as $\epsilon\to0$, there exists 
$\epsilon_0>0$ such that for $0<\epsilon\le\epsilon_0$, 
\begin{eqnarray*}
\|\phi-
\phi_{\epsilon}\|_{H_x^m}<\frac{\eta}{4},
\qquad
\|\tilde{\phi}-
\tilde{\phi}_{\epsilon}\|_{H_x^m}<\frac{\eta}{4}.
\end{eqnarray*}
Further choosing $\epsilon_0$ sufficiently 
small so that $C\epsilon_0^{m-3}<\eta/4$, 
we have that if $\tilde{\phi}\in H^m(\ttt)$ 
satisfies $\|\phi-\tilde{\phi}\|_{H_x^m}<\delta$, 
then taking $0<\epsilon\le\epsilon_0$ in (\ref{k5}), 
we have
\begin{eqnarray*}
\sup_{t\in[0,T)}\|S_t(\phi)-S_t(\tilde{\phi})\|_{H_x^m}
<\frac{\eta}{4}+\frac{\eta}{4}+\frac{\eta}{4}+\frac{\eta}{4}
=\eta.
\end{eqnarray*}
The proof of Theorem \ref{well-posedness} 
is now complete. 

\vglue 1\baselineskip
\noindent
{\bf Acknowledgments.}
The authors would like to thank Dr. Masaya Maeda for fruitful comments. 



\begin{thebibliography}{30}

\bibitem{ABFS} L. Abdelouhab, J.L. Bona, 
M. Felland and J.-C. Saut, 
\textit{Nonlocal models for nonlinear dispersive waves}. 
Physica D {\bf 40} (1989), 360--392.


\bibitem{BS} J.L. Bona and R. Smith, 
\textit{The initial-value problem for the Korteweg-de Vries equation}. 
Philos. Trans. Roy. Soc. London Ser. A {\bf 278} (1975), 555--601.

\bibitem{B} J. Bourgain, \textit{Fourier restriction 
phenomena for certain lattice subsets and applications 
to nonlinear evolution equations. I Schr\"odinger equations,
 II The KdV equation}. Geom, Funct. Anal. {\bf 3} (1993), 
107-156, 209--262.

\bibitem{C}  H. Chihara, \textit{The initial value problem for 
Schr\"{o}dinger equations on the torus}. Int. Math. Res. Not. 
{\bf 2002} (2002) 789--820.

\bibitem{Da} L. S. Da Rios, \textit{On the motion of an 
unbounded fluid with a vortex filament
of any shape}. [in Italian] Rend, Circ. Mat. 
Palermo. {\bf 22} (1906), 117--135.

\bibitem{D} S. Doi, \textit{Smoothing effects of Schr\"{o}dinger 
evolution groups on Riemannian manifolds}. 
Duke Math. J. {\bf 82} (1996), no. 3, 679--706. 

\bibitem{FM} Y. Fukumoto and H. K. Moffatt, 
{\sl Motion and expansion of a
viscous vortex ring. Part I. A higher-order 
asymptotic formula for the
velocity}. J. Fluid. Mech. {\bf 417} (2000), 1--45.

\bibitem{GH} 
A. Gr\"{u}nrock and S. Herr, 
\textit{Low regularity local well-posedness of the derivative nonlinear 
Schr\"{o}dinger equation with periodic initial data}. 
SIAM J. Math. Anal. {\bf 39} (2008), 1890--1920. 

\bibitem{Hasimoto} H. Hasimoto ,
\textit{A soliton on a vortex filament}. J. Fluid Mech. {\bf 51}
(1972), 477--485.

\bibitem{H} 
S. Herr, \textit{On the Cauchy problem for the derivative nonlinear 
Schr\"{o}dinger equation with periodic boundary condition}. 
Int. Math. Res. Not. {\bf 2006} Art. ID 96763 (2006), 33 pp.

\bibitem{HJ1} Z. Huo and Y. Jia, \textit{The Cauchy problem for the fourth-order 
nonlinear Schr\"{o}dinger equation related to the vortex filament}. 
J. Differential Equations {\bf 214} (2005), 1--35.

\bibitem{HJ2} Z. Huo and Y. Jia, \textit{A refined well-posedness for the 
fourth-order nonlinear Schr\"{o}dinger equation related to the vortex filament}. 
Comm. Partial Differential Equations {\bf 32} (2007), no. 10-12, 1493--1510.

\bibitem{KPV1} C. E. Kenig, G. Ponce and L. Vega, 
\textit{Oscillatory integrals and regularity of 
dispersive equations}. Indiana Univ. math J. 
{\bf 40} (1991), 33--69.

\bibitem{KPV2} C. E. Kenig, G. Ponce and L. Vega, 
\textit{The Cauchy problem for the Korteweg-de Vries 
equation in Sobolev spaces of negative indices}, 
Duke Math J. {\bf 71} (1993), 1--21.


\bibitem{KPV3} C. E. Kenig, G. Ponce and L. Vega, 
\textit{A bilinear estimate with applications to the 
KdV equation}. J. Amer. Math. Soc. {\bf 9} (1996), 573--603.


\bibitem{Kida} S. Kida, \textit{A vortex filament moving without change of form}. 
J. Fluid Mech. {\bf 112} (1981), 397--409.

\bibitem{Kwon} S. Kwon, \textit{On the fifth-order KdV equation: 
local well-posedness and lack of uniform continuity 
of the solution map}, J. Differential Equations 
{\bf 245} (2008),  2627--2659.

\bibitem{LP} J. Langer and R. Perline, 
\textit{Poisson geometry of the filament equation}, 
J. Nonlinear Sci. {\bf 1} (1991), 71--93.

\bibitem{M} R. Mizuhara, \textit{The initial value problem 
for third and fourth order dispersive equations 
in one space dimension}, Funkcial. Ekvac. {\bf 49} (2006), 1--38.

\bibitem{MS} M. Maeda and J. Segata, \textit{Existence and 
Stability of standing waves of fourth order 
nonlinear Schr\"odinger type equation related to vortex filament}. 
Funkcialaj Ekvacioj {\bf 54} (2011) 1--14.

\bibitem{MV} A. Moyua and L. Vega, \textit{Bounds for 
the maximal function associated to periodic solutions 
of one-dimensional dispersive equations}, 
Bull. Lond. Math. Soc. {\bf 40} (2008), 117--128.

\bibitem{Sch} M. Jr. Schwarz, 
\textit{The initial value problem for the sequence of generalized Korteweg-de Vries equations}. 
Adv. in Math. {\bf 54} (1984), 22--56.

\bibitem{S1} J. Segata, \textit{Well-posedness for the fourth order
nonlinear Schr\"odinger type equation related to the vortex
filament}, Diff. and Integral Eqs. {\bf 6} (2003), 841--864.

\bibitem{S2} J. Segata, 
\textit{Remark on well-posedness for the fourth order
 nonlinear Schr\"{o}dinger type equation}, 
Proc. Amer. Math. Soc. {\bf 132}
(2004), 3559--3568.

\bibitem{S3} J. Segata, 
\textit{Well-posedness and existence of standing waves for the fourth order nonlinear 
Schr\"{o}dinger type equation}. Discrete Contin. Dyn. Syst. 
{\bf 27} (2010), 1093?1105.

\bibitem{TF1} M. Tsutsumi and I. Fukuda, 
\textit{On solutions of the derivative nonlinear Schr\"{o}dinger equation. 
Existence and uniqueness theorem}. Funkcial. Ekvac. {\bf 23} (1980), 259--277.

\bibitem{TF2} M. Tsutsumi and I. Fukuda, 
\textit{On solutions of the derivative nonlinear Schr\"{o}dinger equation. II}. 
Funkcial. Ekvac. {\bf 24} (1981), 85--94. 

\end{thebibliography}
\end{document}